\documentclass[reqno]{amsart}
\usepackage{amssymb,amsmath,amsfonts,amsthm,latexsym}
\newtheorem{thm}{Theorem}[section]

\newtheorem{lem}[thm]{Lemma}
\newtheorem{prop}[thm]{Proposition}
\newtheorem{defn}[thm]{Definition}
\newtheorem{rem}[thm]{Remark}
\newtheorem{ex}[thm]{Example}
\numberwithin{equation}{section}
\DeclareMathOperator{\lin}{lin}

\begin{document}

\title{The classification of non-characteristically nilpotent
 filiform Leibniz algebras}
\author{ A.Kh. Khudoyberdiyev, M. Ladra and B.A. Omirov}
\address{[M. Ladra] Department of Algebra, University of Santiago de Compostela, 15782, Spain.}
\email {manuel.ladra@usc.es}
\address{[A.Kh. Khudoyberdiyev -- B.A. Omirov] Institute of Mathematics, 29,
Do'rmon yo'li srt., 100125, Tashkent (Uzbekistan)}
\email{khabror@mail.ru
--- omirovb@mail.ru}



\begin{abstract}
In this paper we investigate the derivations of filiform Leibniz
algebras. Recall that the set of filiform Leibniz algebras of
fixed dimension is decomposed into three non-intersected families.
We found sufficient conditions under which filiform Leibniz
algebras of the first family are characteristically nilpotent.
Moreover, for the first family we classify non-characteristically
nilpotent algebras by means of Catalan numbers. In addition, for
the rest two families of filiform Leibniz algebras we describe
non-characteristically nilpotent algebras, i.e., those filiform
Leibniz algebras which lie in the complementary set to those
characteristically nilpotent.
\end{abstract}

\maketitle

% ----------------------------------------------------------------

\section{Introduction} \label{S:intro}

In 1955, Jacobson \cite{Jac2} proved that every Lie algebra over a
field of characteristic zero admitting a non-singular derivation
is nilpotent. The problem whether the inverse of this statement is
correct remained open until that an example of an 8-dimensional
nilpotent Lie algebra whose all derivations are nilpotent was
constructed in \cite{DiLi}. They called such type of algebras
characteristically nilpotent Lie algebras.

If all derivations of an algebra are nilpotent (inner derivations
are nilpotent, as well), then by Engel's theorem we conclude that
a characteristically nilpotent Lie algebra is nilpotent. Inverse
statement is not true, because there exist nilpotent Lie algebras
admitting non-nilpotent derivations. Therefore, the subset of
characteristically nilpotent Lie algebras is strictly embedded
into the set of nilpotent Lie algebras.

Papers \cite{CaNu,Kha1,LeTo} and
others are devoted to the investigation of characteristically nilpotent Lie algebras.
 The classification of nilpotent Lie algebras till
dimension 8 shows that there is no characteristically nilpotent
Lie algebras in dimensions less than 7. Moreover, it is shown
that there exist characteristically nilpotent Lie algebras in each
dimension from 7 till 13-dimensional. Taking into account that
a direct sum of characteristically nilpotent Lie algebras is
characteristically nilpotent, then we have the existence of
characteristically nilpotent Lie algebras in each finite dimension
starting from 7.

It was conjectured for a long time that there are ``a few'' algebras
of this kind, and only in \cite{Kha2}, it was proved that
every irreducible component of the variety of complex filiform Lie
algebras of dimension greater than 7 contains a Zariski open
set, consisting of characteristically nilpotent Lie algebras. This
implies that there are ``many'' characteristically nilpotent Lie
algebras, and hence they play an important role in the description
of the variety of nilpotent Lie algebras.

The notion of Leibniz algebra has been introduced in \cite{Lod}
as a non-antisymmetric generalization of Lie algebras. During the
last 20 years the theory of Leibniz algebras has been actively
studied and many results of the theory of Lie algebras have been
extended to Leibniz algebras (see, e.g. \cite{AAOK,AyOm2,GoOm}). In particular, an analogue of
Jacobson's theorem was proved for Leibniz algebras \cite{LaRiTu}.
Moreover, it is shown that similarly to the case of Lie algebras
for Leibniz algebras the inverse of Jacobson's statement does not
hold. In \cite{Omi1}, analogously as for Lie algebras,  the notion of characteristically nilpotent Leibniz algebra
was defined and some families of characteristically nilpotent
filiform Leibniz algebras were found. Moreover, there was presented a criterion
of characteristically nilpotency of some filiform Leibniz
algebras. Due to existence of a example of a characteristically
nilpotent Leibniz algebra which does not satisfy the condition of
\cite{Omi1}, the criterion is not correct.

It is known that the class of all filiform Leibniz algebras is
split into three non-intersected families \cite{AyOm2,GoOm},
where one of the family contains filiform Lie algebras and the
other two families come out from naturally graded non-Lie
filiform Leibniz algebras. A isomorphism criteria for these two families of filiform
Leibniz algebras have been given in \cite{GoOm}.

In this paper, as opposed to \cite{Omi1}, we find out a correct
criterion of characteristically nilpotency of filiform Leibniz
algebras. In addition, it is described, up to isomorphism, the
class of filiform Leibniz algebras complementary to
characteristically nilpotent filiform Leibniz algebras. Note that
filiform Leibniz algebras were classified only up to dimension
less than $10$ in \cite{GoJiKh,Omi2,RaSo,Rik}. Here we classify non-characteristically nilpotent
non-Lie filiform Leibniz algebras for any fixed dimension.
Recall that non-characteristically nilpotent filiform Lie algebras
are described in \cite{GoHa}.

The classification of non-characteristically nilpotent Leibniz
algebras plays an important role in structure theory of solvable
Leibniz algebras. In theory of finite dimensional Leibniz algebras
it is known the description of solvable Leibniz algebras with a
given nilradical based on properties of non-nilpotent derivations
of the nilradical. Hence, solvable Leibniz algebras can have only
non-characteristically nilpotent nilradical. Therefore, it is very
crucial to indicate non-characteristically nilpotent Leibniz
algebras. Papers \cite{NdWi,SnWi,WaLiDe} are devoted  to classifications of solvable Lie algebras with various
types of nilradical. The solvable Leibniz algebras with null-filiform
and naturally graded filiform nil-radical are classified in
\cite{CLOK1,CLOK2}.

Catalan numbers are a well-known sequence of numbers and they are
involved in a lot of branches of mathematics (combinatorics, graph
theory, probability theory and many others). In the present paper
we classify some kinds of non-characteristically nilpotent
filiform Leibniz algebras in terms of $p$-th Catalan numbers.

In order to achieve our goal, we have organized the paper as follows: in Section
\ref{S:prel} we present necessary definitions and results that will be used in the rest of the paper.
 In Section \ref{S:main} we describe characteristically nilpotent filiform
non-Lie Leibniz algebras and give the classification of
non-characteristically nilpotent filiform non-Lie Leibniz
algebras.

Through the paper all the spaces and algebras are assumed finite
dimensional.

\section{Preliminaries}\label{S:prel}

In this section we give necessary definitions and preliminary
results.

\begin{defn} An algebra $(L,[-,-])$ over a field $F$ is called a Leibniz algebra if for any $x,y,z\in L$, the so-called Leibniz identity
\[ \big[[x,y],z\big]=\big[[x,z],y\big]+\big[x,[y,z]\big] \] holds.
\end{defn}

For a Leibniz algebra $L$ consider the following central lower
series:
\[
L^1=L,\qquad L^{k+1}=[L^k,L^1] \qquad k \geq 1.
\]

\begin{defn} A Leibniz algebra $L$ is called
nilpotent if there exists  $s\in\mathbb N $ such that $L^s=0$.
\end{defn}

\begin{defn} A Leibniz algebra $L$ is said to be filiform if
$\dim L^i=n-i$, where $n=\dim L$ and $2\leq i \leq n$.
\end{defn}

The following theorem decomposes  all $(n+1)$-dimensional
filiform Leibniz algebras into  three families of algebras.

\begin{thm}[\cite{OmRa}]
Any complex $(n+1)$-dimensional filiform Leibniz algebra admits a
basis $\{e_0, e_1, \dots, e_n\}$ such that the table of multiplication
of the algebra has one of the following forms:
\[F_1(\alpha_3,\alpha_4, \dots, \alpha_n, \theta ):
\left\{\begin{array}{ll}
[e_0,e_0]=e_{2},&  \\[1mm]
[e_i,e_0]=e_{i+1}, & \  1\leq i \leq {n-1},\\[1mm]
[e_0,e_1]=\sum\limits_{k=3}^{n-1}\alpha_{k}e_k+\theta e_n, & \\[1mm]
[e_i,e_1]=\sum\limits_{k=i+2}^n\alpha_{k+1-i}e_k, & \ 1\leq i \leq
{n-2},
\end{array} \right.\]

\[F_2(\beta_3,\beta_4, \dots, \beta_n, \gamma ):
\left\{\begin{array}{ll}
[e_0,e_0]=e_{2}, &  \\[1mm]
[e_i,e_0]=e_{i+1}, & \  2\leq i \leq {n-1},\\[1mm]
[e_0,e_1]= \sum\limits_{k=3}^n\beta_{k}e_k, & \\[1mm]
[e_1,e_1]=\gamma e_n, & \\[1mm]
[e_i,e_1]=\sum\limits_{k=i+2}^n\beta_{k+1-i}e_k, & \ 2\leq i \leq
{n-2},
\end{array} \right.\]

\[F_3(\theta_1, \theta_2, \theta_3)=\left\{\begin{array}{lll} [e_i,e_0]=e_{i+1}, &
1\leq i \leq {n-1},\\[1mm]
[e_0,e_i]=-e_{i+1}, & 2\leq i \leq {n-1}, \\[1mm]
[e_0,e_0]=\theta_1e_n, &   \\[1mm]
[e_0,e_1]=-e_2+\theta_2e_n, & \\[1mm]
[e_1,e_1]=\theta_3e_n, &  \\[1mm]
[e_i,e_j]=-[e_j,e_i] \in \lin<e_{i+j+1}, e_{i+j+2}, \dots, e_n>, &
1\leq i \leq n-2,\\[1mm]
& 2 \leq j \leq {n-i},\\[1mm]
[e_i,e_{n-i}]=-[e_{n-i},e_i]=\alpha (-1)^{i}e_n, & 1\leq i\leq
n-1,
\end{array} \right.\]
where $\alpha\in\{0,1\}$ for odd $n$ and $\alpha=0$ for even $n$.
Moreover, the structure constants of an algebra from
$F_3(\theta_1, \theta_2, \theta_3)$ should satisfy the Leibniz
identity.
\end{thm}

It is easy to see that algebras of the first and the second
families are non-Lie algebras. Note that if $(\theta_1, \theta_2,
\theta_3) = (0,0,0)$, then an algebra of the third class is a Lie
algebra and if $(\theta_1, \theta_2, \theta_3) \neq (0,0,0)$, then
it is a non-Lie Leibniz algebra.

Further we will use the following lemma.
\begin{lem}[\cite{GoOm}]\label{lem2} For any ${0\leq p\leq {n-k}, \ 3\leq k\leq n}$, the following equality holds:
\[
 \sum\limits_{i=k}^n{a(i)}\sum\limits_{j=i+p}^n{b(i,j)e_j}=
\sum\limits_{j=k+p}^n{\sum\limits_{i=k}^{j-p}{a(i)b(i,j)e_j}}.
\]
\end{lem}

Let us present a isomorphism criterion for the first and second
families of non-Lie filiform Leibniz algebras.

\begin{thm}[\cite{GoOm}]\label{thm2} \

(a) Two algebras  from  the families $F_1( \alpha_3, \alpha_4, \dots,
\alpha_n, \theta )$ and $F_1^{'}(\alpha'_{3}, \alpha'_{4}, \dots,
\alpha'_{n}, \theta')$ are isomorphic if and only if there exist
$A,B\in \mathbb{C}$ such that $A(A+B)\neq 0$ and the following
conditions hold:
{\small
\begin{align*}
\alpha'_3& =\frac{(A+B)}{A^2} \alpha_3, \\
\alpha'_t & =\frac{1}{A^{t-1}} \left((A+B)\alpha_t -
\sum\limits_{k=3}^{t-1}\left(
 \binom {k-1}{k-2} \,
   A^{k-2}B\alpha_{t+2-k} +\right. \right. +
\binom {k-1}{k-3} \, A^{k-3}B^2\sum\limits_{i_1=k+2}^{t}\alpha_{t+3-i_1}
\alpha_{i_1+1-k} \\
& {}  + \binom {k-1}{k-4} \,
 A^{k-4}B^3\sum\limits_{i_2=k+3}^{t}
\sum\limits_{i_1=k+3}^{i_2}\alpha_{t+3-i_2}\cdot
\alpha_{i_2+3-i_1}\cdot
\alpha_{i_1-k} + \cdots \\
& {} + \binom {k-1}{1} \,
 AB^{k-2}\sum\limits_{i_{k-3}=2k-2}^{t}
\sum\limits_{i_{k-4}=2k-2}^{i_{k-3}} \dots
\sum\limits_{i_1=2k-2}^{i_2} \alpha_{t+3-i_{k-3}}\cdot
\alpha_{i_{k-3}+3-i_{k-4}}\cdot \dots \cdot
\alpha_{i_2+3-i_1}\alpha_{i_1+5-2k}\\
& {}  + \left.\left. B^{k-1}\sum\limits_{i_{k-2}=2k-1}^{t}
\sum\limits_{i_{k-3}=2k-1}^{i_{k-2}} \dots
\sum\limits_{i_1=2k-1}^{i_2} \alpha_{t+3-i_{k-2}}\cdot
\alpha_{i_{k-2}+3-i_{k-3}}\cdot \dots \cdot \alpha_{i_2+3-i_1}\cdot
\alpha_{i_1+4-2k}\right)\cdot \alpha'_{k}\right),
\end{align*}
}
{\small
\begin{align*}
\theta'& =\frac{1}{A^{n-1}}\left(A\theta +B\alpha_n-
\sum\limits_{k=3}^{n-1}\left(
\binom {k-1}{k-2} \,
 A^{k-2}B\alpha_{n+2-k}+\right.\right.
\binom {k-1}{k-3} \, A^{k-3}B^2\sum\limits_{i_1=k+2}^{n}\alpha_{n+3-i_1}
\alpha_{i_1+1-k} \\
& {}  + \binom {k-1}{k-4} \,  A^{k-4}B^3\sum\limits_{i_2=k+3}^{n}
\sum\limits_{i_1=k+3}^{i_2}\alpha_{n+3-i_2} \alpha_{i_2+3-i_1}
\alpha_{i_1-k} + \cdots  \\
& {}  + \binom {k-1}{1} \,  AB^{k-2}\sum\limits_{i_{k-3}=2k-2}^{n}
\sum\limits_{i_{k-4}=2k-2}^{i_{k-3}} \dots
\sum\limits_{i_1=2k-2}^{i_2} \alpha_{n+3-i_{k-3}}
\alpha_{i_{k-3}+3-i_{k-4}} \dots \alpha_{i_2+3-i_1}
\alpha_{i_1+5-2k}  \\
& {}  + \left.\left. B^{k-1}\sum\limits_{i_{k-2}=2k-1}^{n}
\sum\limits_{i_{k-3}=2k-1}^{i_{k-2}} \dots
\sum\limits_{i_1=2k-1}^{i_2} \alpha_{n+3-i_{k-2}}\cdot
\alpha_{i_{k-2}+3-i_{k-3}}\cdot \dots \cdot \alpha_{i_2+3-i_1}\cdot
\alpha_{i_1+4-2k}\right)\cdot \alpha'_{k}\right),
\end{align*}
}
where $4\leq t\leq n$.

(b) Two algebras from the families $F_2(\beta_3, \beta_4, \dots,
\beta_n, \gamma )$ and $F_2^{'}(\beta'_{3}, \beta'_{4}, \dots,
\beta'_{n}, \gamma')$ are isomorphic if and only if there exist
$A,B, D \in \mathbb{C}$ such that $AD\neq 0$  and the following
conditions hold:
{\small
\begin{align*}
\gamma'&=\frac{D^2}{A^{n}} \gamma, \\
\beta'_3 &=\frac{D}{A^2} \beta_3, \\
\beta'_t & =\frac{1}{A^{t-1}}\left(D\beta_t - \sum\limits_{k=3}^{t-1}\left(
\binom {k-1}{k-2} \, A^{k-2}B\beta_{t+2-k}+ \binom {k-1}{k-3} \,
A^{k-3}B^2\sum\limits_{i_1=k+2}^{t}\beta_{t+3-i_1}
\cdot \beta_{i_1+1-k}  \right. \right.\\
& {} + \binom {k-1}{k-4} \, A^{k-4}B^3\sum\limits_{i_2=k+3}^{t}
\sum\limits_{i_1=k+3}^{i_2}\beta_{t+3-i_2}\cdot
\beta_{i_2+3-i_1}\cdot
\beta_{i_1-k} + \ \cdots \  \\
& {} + \binom {k-1}{1} \, AB^{k-2}\sum\limits_{i_{k-3}=2k-2}^{t}
\sum\limits_{i_{k-4}=2k-2}^{i_{k-3}} \dots
\sum\limits_{i_1=2k-2}^{i_2} \beta_{t+3-i_{k-3}}
\beta_{i_{k-3}+3-i_{k-4}} \dots \beta_{i_2+3-i_1}\beta_{i_1+5-2k}\\
& {} +  \left. \left. B^{k-1}\sum\limits_{i_{k-2}=2k-1}^{t}
\sum\limits_{i_{k-3}=2k-1}^{i_{k-2}} \dots
\sum\limits_{i_1=2k-1}^{i_2} \beta_{t+3-i_{k-2}}
\beta_{i_{k-2}+3-i_{k-3}} \dots
\beta_{i_2+3-i_1}\beta_{i_1+4-2k}\right)\beta'_k\right),
\end{align*}
}
where $4\leq t\leq n-1$,
{\small
\begin{align*}
 \beta'_n & = \frac{BD\gamma}{A^{n}}
+\frac{1}{A^{n-1}}\left(D\beta_n - \sum\limits_{k=3}^{n-1}\left(
\binom {k-1}{k-2} \,
A^{k-2}B\beta_{n+2-k} + \binom {k-1}{k-3} \,
A^{k-3}B^2\sum\limits_{i_1=k+2}^{n}\beta_{n+3-i_1}
\cdot \beta_{i_1+1-k}  \right. \right.\\
& {} + \binom {k-1}{k-4} \, A^{k-4}B^3\sum\limits_{i_2=k+3}^{n}
\sum\limits_{i_1=k+3}^{i_2}\beta_{n+3-i_2}\cdot
\beta_{i_2+3-i_1}\cdot
\beta_{i_1-k} + \ \cdots  \  \\
& {} + \binom {k-1}{1} \, AB^{k-2}\sum\limits_{i_{k-3}=2k-2}^{n}
\sum\limits_{i_{k-4}=2k-2}^{i_{k-3}} \dots
\sum\limits_{i_1=2k-2}^{i_2} \beta_{n+3-i_{k-3}}
\beta_{i_{k-3}+3-i_{k-4}} \dots \beta_{i_2+3-i_1}\beta_{i_1+5-2k}\\
& {} + \left. \left. B^{k-1}\sum\limits_{i_{k-2}=2k-1}^{n}
\sum\limits_{i_{k-3}=2k-1}^{i_{k-2}} \dots
\sum\limits_{i_1=2k-1}^{i_2} \beta_{n+3-i_{k-2}}
\beta_{i_{k-2}+3-i_{k-3}} \dots
\beta_{i_2+3-i_1}\beta_{i_1+4-2k}\right)\beta'_k\right),
\end{align*} }
where $\binom{m}{n}$ are the binomial
coefficients.
\end{thm}

The derivation of Leibniz algebras is defined as usual.

\begin{defn}
A linear transformation $d$ of a Leibniz algebra $L$ is called a
derivation if for any $x, y\in L$
\[d([x,y])=[d(x),y]+[x, d(y)].\]
\end{defn}

A nilpotent Leibniz algebra is called \emph{characteristically
nilpotent} if all its  derivations are nilpotent. As it was
mentioned in Section \ref{S:intro}, the class of characteristically
nilpotent Leibniz algebras is a subclass of the nilpotent Leibniz
algebras.

In  \cite{Omi1} the following isomorphism
criterion is obtained.

\begin{thm}[\cite{Omi1}] \label{tnil} A Leibniz algebra of the family $F_1(\alpha_3,\alpha_4, \dots, \alpha_n, \theta
)$ is characteristically nilpotent if there exist $i, j  \ (3
\leq i \neq j \leq n)$ such that $\alpha_i \alpha_j \neq 0$.
\end{thm}

Further we shall need the  notion of Catalan numbers. The Catalan
numbers are defined as follows:
\[C_n= \frac {1} {n+1} \ \binom{2n}{n} = \frac{(2n)!} {(n+1)!n!}.\]

The generalized Catalan numbers or $p$-th Catalan numbers were
defined in \cite{HiPe} by the formula  \[C^{p}_{n} = \frac {1}
{(p-1)n+1} \ \binom{pn}{n}\,.\]

Obviously, 2-th Catalan numbers are usual Catalan numbers.

H. W. Gould  developed a generalization of the $n$-th Catalan numbers, also called
Rothe numbers or Rothe/Hagen coefficients of the first type (see \cite{Gou}),
 as follows:
\[A_n(x, z) = \frac x {x+zn} \ \binom{x+zn}{n} ,\]
together with their convolution formula
\begin{equation} \label{E:conv}
\sum\limits_{k=0}^n A_n(x, z) A_{n-k}(y, z) = A_n(x+y, z).
\end{equation}

Note that $A_n(1, p)$ is the $p$-th Catalan number $C^{p}_{n}$.

From the convolution formula \eqref{E:conv}, it is not difficult to obtain
the following formula:
\begin{equation} \label{E:comb}
\sum\limits_{k=1}^{n} C_k^p C_{n-k}^p = \frac {2n}
{(p-1)n+p+1}C_{n+1}^p\,.
\end{equation}

\section{The main results}\label{S:main}

Since filiform characteristically nilpotent Lie algebras are
already in detail studied in \cite{Kha1,Kha2}, we shall
consider only non-Lie Leibniz algebras.

In this section we describe characteristically nilpotent filiform
non-Lie Leibniz algebras and give the classification of
non-characteristically nilpotent filiform non-Lie Leibniz
algebras.

\subsection{Characteristically nilpotent filiform Leibniz algebras of the family $F_1( \alpha_3, \alpha_4, \dots,
\alpha_n, \theta )$}

\

Let $L$ be a filiform Leibniz algebra from the family
$F_1(\alpha_3, \alpha_4, \dots, \alpha_n, \theta )$. The following
proposition describes the derivations of such algebras.

\begin{prop} \label{pr1} The derivations of the filiform Leibniz algebras from the family $F_1( \alpha_3, \alpha_4, \dots,
\alpha_n, \theta )$ have the following form:
\[\begin{pmatrix}
 a_0&a_1&a_2&a_3&\dots&a_{n-2}&a_{n-1}&a_n\\
0&a_0+a_1&a_2&a_3&\dots&a_{n-2}&b_{n-1}&b_n\\
 0& 0& 2a_0+a_1&a_2+a_1\alpha_3&\dots&a_{n-3}+a_1\alpha_{n-2}&
a_{n-2}+a_1\alpha_{n-1}&a_{n-1}+a_1\alpha_n\\
0& 0& 0& 3a_0+a_1& \dots&a_{n-4}+2a_1\alpha_{n-2}&
a_{n-3}+2a_1\alpha_{n-1}&a_{n-2}+2a_1\alpha_{n-1}\\
 \vdots& \vdots& \vdots& \vdots&\vdots& \vdots&\vdots &\vdots \\
0& 0& 0& 0&\dots& (n-2)a_0+a_1&a_2+(n-3)a_1\alpha_3 & a_3+ (n-3)a_1\alpha_4\\
 0& 0& 0& 0&\dots& 0&(n-1)a_0+a_1 & a_2+(n-2)a_1\alpha_3\\
 0& 0& 0& 0&\dots& 0& 0& na_0+a_1
\end{pmatrix}\]
where
\[a_0(\theta - \alpha_n) = 0, \quad a_1(\alpha_n - \theta) =
a_{n-1} - b_{n-1}, \quad
 \alpha_3(a_1-a_0)=0,\]
\begin{equation}\label{E:der}
\alpha_k(a_1-(k-2)a_0)=\frac k 2 a_1\sum\limits_{j=4}^{k}\alpha_{j-1}\alpha_{k-j+3}, \quad 4 \leq k \leq n.
\end{equation}
\end{prop}
\begin{proof} Let $L$ be a filiform Leibniz algebra from the family $F_1( \alpha_3, \alpha_4, \dots,
\alpha_n, \theta )$ and let $d: L \rightarrow L$ be a derivation
of $L$.

Put
\[d(e_0) = \sum\limits_{k=0}^na_ke_k, \quad d(e_1) = \sum\limits_{k=0}^nb_ke_k.\]

By the property of the derivation, we have
\begin{align*}
d(e_2) & = d([e_0, e_0]) = [d(e_0), e_0] +  [e_0, d(e_0)] =
\big[\sum\limits_{k=0}^na_ke_k, e_0\big]+\big[e_0, \sum\limits_{k=0}^na_ke_k\big]\\
& {} =(a_0+a_1)e_2 + \sum\limits_{k=3}^na_{k-1}e_k+ a_0e_2 + a_1\Big(\sum\limits_{k=3}^{n-1}\alpha_ke_k+ \theta e_n\Big)\\
&{}=(2a_0+a_1)e_2 + \sum\limits_{k=3}^{n-1}(a_{k-1} + a_1\alpha_{k})e_k + (a_{n-1}+ a_1\theta) e_n.
\end{align*}

By induction we derive
\[d(e_i)=(ia_0+a_1)e_i + \sum\limits_{k=i+1}^n(a_{k-i+1}+(i-1)a_1\alpha_{k-i+2})e_k, \quad 3 \leq i \leq n.\]

Indeed, if the induction hypothesis is  true for $i$, then for $i+1$
it implies, from the following chain of equalities:
\begin{align*}
d(e_{i+1})& = d([e_i, e_0]) = [d(e_i), e_0] +  [e_i, d(e_0)] \\
& {} =\big[(ia_0+a_1)e_i + \sum\limits_{k=i+1}^n(a_{k-i+1}+(i-1)a_1\alpha_{k-i+2})e_k, e_0\big]+\big[e_i, \sum\limits_{k=0}^na_ke_k\big] \\
&{} =(ia_0+a_1)e_{i+1} + \sum\limits_{k=i+2}^n(a_{k-i}+(i-1)a_1\alpha_{k-i+1})e_k +
a_0e_{i+1} + a_1\sum\limits_{k=i+2}^n\alpha_{k-i+1}e_k \\
& {} =\big((i+1)a_0+a_1\big)e_{i+1} + \sum\limits_{k=i+2}^n(a_{k-i}+ia_1\alpha_{k-i+1})e_k.
\end{align*}

Consider the property of derivation:
\begin{align*}
d([e_1, e_0]) &= [d(e_1), e_0] +  [e_1, d(e_0)] = \big[\sum\limits_{k=0}^nb_ke_k, e_0\big]+\big[e_1, \sum\limits_{k=0}^na_ke_k\big]  \\
& {} =(b_0+b_1)e_2 + \sum\limits_{k=3}^nb_{k-1}e_k+ a_0e_2 + a_1\sum\limits_{k=3}^{n}\alpha_ke_k
=(a_0+b_0+b_1)e_2 + \sum\limits_{k=3}^{n}(b_{k-1} + a_1\alpha_{k})e_k.
\end{align*}
On the other hand
\[d([e_1, e_0]) = d(e_2) = (2a_0+a_1)e_2 + \sum\limits_{k=3}^{n-1}(a_{k-1} + a_1\alpha_{k})e_k + (a_{n-1}+ a_1\theta) e_n.\]

Comparing the coefficients at the basic elements we have
\[b_0 + b_1=a_0+a_1, \quad b_i = a_i,\quad  2 \leq i \leq n-2, \quad a_1(\alpha_n - \theta) = a_{n-1} - b_{n-1}.\]
Using Lemma \ref{lem2} we obtain
\begin{equation*}
\begin{split}
& d([e_0, e_1]) = [d(e_0), e_1] +  [e_0, d(e_1)] = \big[\sum\limits_{k=0}^na_ke_k, e_1\big]+\big[e_0, \sum\limits_{k=0}^nb_ke_k\big]  \\
& {} = a_0\Big(\sum\limits_{k=3}^{n-1}\alpha_ke_k + \theta e_n\Big) +
\sum\limits_{k=1}^{n-2}a_k\sum\limits_{j=k+2}^n\alpha_{j-k+1}e_j+b_0e_2+b_1\Big(\sum\limits_{k=3}^{n-1}\alpha_ke_k
+ \theta e_n\Big)\\
& {}= b_0e_2+(a_0+b_1)\alpha_3e_3+\sum\limits_{k=4}^{n-1}(a_0+b_1)\alpha_ke_k + (a_0+b_1)\theta e_n +
a_1\alpha_3e_3+a_1\sum\limits_{j=4}^{n}\alpha_je_j+\sum\limits_{k=4}^{n}a_{k-2}\sum\limits_{j=k}^n\alpha_{j-k+3}e_j\\
& {} = b_0e_2+(a_0+a_1+b_1)\alpha_3e_3+\sum\limits_{k=4}^{n-1}(a_0+a_1+b_1)\alpha_ke_k + \big((a_0+b_1)\theta +a_1\alpha_n\big)e_n +
\sum\limits_{k=4}^{n}\Big(\sum\limits_{j=4}^k
a_{j-2}\alpha_{k-j+3}\Big)e_k\\
& {} = b_0e_2+(a_0+a_1+b_1)\alpha_3e_3+\sum\limits_{k=4}^{n-1}\Big((a_0+a_1+b_1)\alpha_k + \sum\limits_{j=4}^k
a_{j-2}\alpha_{k-j+3}\Big)e_k\\
& {}   \qquad \quad  + \Big((a_0+b_1)\theta +a_1\alpha_n + \sum\limits_{j=4}^n
a_{j-2}\alpha_{n-j+3}\Big)e_n.
\end{split}
\end{equation*}

On the other hand
\begin{equation*}
\begin{split}
& d([e_0, e_1]) = d\Big(\sum\limits_{k=3}^{n-1}\alpha_ke_k + \theta e_n\Big) = \sum\limits_{k=3}^{n-1}\alpha_k d(e_k) + \theta d(e_n)\\
& {} = \sum\limits_{k=3}^{n-1} \alpha_k\Big( (ka_0+a_1)e_k + \sum\limits_{j=k+1}^n(a_{j-k+1}+(k-1)a_1\alpha_{j-k+2})e_j\Big) + (na_0+a_1)\theta e_n\\
& {} = (3a_0+a_1)\alpha_3e_3+\sum\limits_{k=4}^{n-1} \alpha_k(ka_0+a_1) e_k
 + (na_0+a_1)\theta e_n +\sum\limits_{k=4}^n\Big(\alpha_{k-1}
\sum\limits_{j=k}^n(a_{j-k+2}+(k-2)a_1\alpha_{j-k+3})e_j\Big)\\
& {} = (3a_0+a_1)\alpha_3e_3+\sum\limits_{k=4}^{n-1} \Big( (ka_0+a_1)\alpha_k +
\sum\limits_{j=4}^k a_{j-2}\alpha_{k-j+3}+ \frac k 2
a_1\sum\limits_{j=4}^k \alpha_{j-1}\alpha_{k-j+3}\Big)e_k\\
 & {} \qquad \qquad \qquad \qquad + \Big((na_0+a_1)\theta +
\sum\limits_{j=4}^n a_{j-2}\alpha_{n-j+3}+ \frac n 2
a_1\sum\limits_{j=4}^n \alpha_{j-1}\alpha_{n-j+3} \Big) e_n.
\end{split}
\end{equation*}

Comparing the coefficients at the basic elements we conclude
\[b_0=0 \Rightarrow b_1=a_0+a_1, \quad (a_0+a_1+b_1)\alpha_3 = (3a_0+a_1)\alpha_3,\]
\[(a_0+a_1+b_1)\alpha_k + \sum\limits_{j=4}^k
a_{j-2}\alpha_{k-j+3} = (ka_0+a_1)\alpha_k + \sum\limits_{j=4}^k
a_{j-2}\alpha_{k-j+3}+ \frac k 2 a_1\sum\limits_{j=4}^k
\alpha_{j-1}\alpha_{k-j+3}, \quad 4 \leq k \leq n-1,\]
\[(a_0+b_1)\theta +a_1\alpha_n + \sum\limits_{j=4}^n
a_{j-2}\alpha_{n-j+3} = (na_0+a_1)\theta + \sum\limits_{j=4}^n
a_{j-2}\alpha_{n-j+3}+ \frac n 2 a_1\sum\limits_{j=4}^n
\alpha_{j-1}\alpha_{n-j+3}.\]

Replacing $b_1=a_0+a_1$ we get
\begin{align}\label{E:res}
(a_1-a_0)\alpha_3& =0,  \notag\\
(a_1 - (k-2)a_0)\alpha_k  & = \frac k 2 a_1\sum\limits_{j=4}^k
\alpha_{j-1}\alpha_{k-j+3}, \qquad 4 \leq k \leq n-1, \notag\\
(2-n)a_0\theta +a_1\alpha_n &= \frac n 2 a_1\sum\limits_{j=4}^n
\alpha_{j-1}\alpha_{n-j+3}.
\end{align}

Similarly to the above argumentations we derive
\[d([e_1, e_1]) = [d(e_1), e_1] +  [e_1, d(e_1)]=2(a_0+a_1)\alpha_3e_3+\sum\limits_{k=4}^{n}\Big((2a_0+2a_1)\alpha_k + \sum\limits_{j=4}^k
a_{j-2}\alpha_{k-j+3}\Big)e_k.\]

On the other hand
\begin{align*}
d([e_1, e_1]) &= d\Big(\sum\limits_{k=3}^n\alpha_ke_k\Big) = \sum\limits_{k=3}^n\alpha_k d(e_k)\\
&{} = (3a_0+a_1)\alpha_3e_3+\sum\limits_{k=4}^n \Big( (ka_0+a_1)\alpha_k +
\sum\limits_{j=4}^k a_{j-2}\alpha_{k-j+3}+ \frac k 2
a_1\sum\limits_{j=4}^k \alpha_{j-1}\alpha_{k-j+3}\Big)e_k.
\end{align*}

Comparing the coefficients at the basic elements we obtain
\[(a_1-a_0)\alpha_3=0\]
and the restriction \eqref{E:der}, i.e. $(a_1 - (k-2)a_0)\alpha_k  = \frac
k 2 a_1\sum\limits_{j=4}^k \alpha_{j-1}\alpha_{k-j+3}, \quad 4
\leq k \leq n$.

From \eqref{E:der} for $k=n$ and the restriction \eqref{E:res}, we have $a_0(\theta -
\alpha_n) = 0$.

Considering the properties of the derivation for $d([e_i, e_2]),  \ 3
\leq i \leq n-2$, we have the same restrictions.
\end{proof}

From Proposition \ref{pr1} it is obvious that if there exist the
pair $a_0, a_1$ such that $(a_0, a_1) \neq (0, 0)$ and the
restriction \eqref{E:der} holds, then a filiform Leibniz algebra of the
first family is non-characteristically nilpotent, otherwise is
characteristically nilpotent.

From \cite{CLOK2} and \cite{Omi1} it is known that the naturally
graded filiform Leibniz algebra (the algebra with $\alpha_i =0, 3
\leq i \leq n, \theta =0$) is non-characteristically nilpotent.

\begin{thm} Let $\theta \neq \alpha_n$ and suppose that there exist $\alpha_k \neq 0, \
3 \leq k \leq n$. Then a filiform Leibniz algebra of the family
$F_1(\alpha_3, \alpha_4, \dots, \alpha_n, \theta)$ is
characteristically nilpotent.
\end{thm}
\begin{proof} Note that it is sufficient to prove $a_0 = a_1=0$.

Let $\theta \neq \alpha_n$, then  the restriction \eqref{E:der}
implies that $a_0 =0$ and we get \[\alpha_3a_1=0, \quad
a_1\alpha_k  = \frac k 2 a_1\sum\limits_{j=4}^k
\alpha_{j-1}\alpha_{k-j+3}, \quad 4 \leq k \leq n.\]

If there exist $\alpha_k \neq 0$, then for the first non-zero
$\alpha_k \neq 0$, we have $\alpha_k a_1 = 0$. Hence $a_1=0$.
\end{proof}

From the above theorem we have that an algebra of the class
$F_1(0, 0, \dots, 0, \theta), \ \theta \neq 0$, is
non-characteristically nilpotent.

\begin{rem} Note that in notations of Theorem \ref{thm2}  putting $A=
\sqrt[\uproot{3} n-2]{\theta}$, we conclude that an algebra $F_1(0, 0, \dots, 0,
\theta)$, $\theta \neq 0$, is isomorphic to the algebra $F_1(0, 0,
\dots, 0, 1)$.
\end{rem}

Below, we present an example which shows that Theorem \ref{tnil} is
not true in general.

\begin{ex} Let $L$ be a 6-dimensional filiform Leibniz
algebra and let $\{e_0, e_1, e_2, e_3, e_4, e_5\}$ be a basis of
$L$ with the following multiplication:
\[\left\{\begin{array}{ll}
[e_0,e_0]=e_{2},&  \\[1mm]
[e_i,e_0]=e_{i+1}, & \  1\leq i \leq {4},\\[1mm]
[e_0,e_1]=e_3 - 2e_4+5e_5, & \\[1mm]
[e_1,e_1]= e_3 - 2e_4+5e_5,& \\[1mm]
[e_2,e_1]= e_4 - 2e_5,& \\[1mm]
[e_3,e_1]= e_5, & \\[1mm]
\end{array} \right.\]
(omitted products are equal to zero).
\end{ex}
Clearly, this algebra satisfies the condition of Theorem
\ref{tnil}, but it is non-characteristically nilpotent, because
the derivations of the algebra have the form:
 \[\begin{pmatrix}
    a_1& a_1& a_3&a_4&a_5&a_6\\
    0& 2a_1& a_3&a_4&a_5&b_6\\
    0& 0& 3a_1&a_1 + a_3&-2a_1+a_4&5a_1+a_5\\
    0& 0& 0& 4a_1&2a_1 + a_3&-4a_1+a_4\\
    0& 0& 0& 0& 5a_1&3a_1 + a_3\\
    0& 0& 0& 0& 0& 6a_1
 \end{pmatrix}.\]

Let us now consider the case $\alpha_n=\theta$.

\begin{lem}\label{l4} Let $L$ be a non-characteristically nilpotent filiform Leibniz
algebra from the family $F_1(\alpha_3, \alpha_4, \dots, \alpha_n,
\alpha_n)$ and let $\alpha_{s} \neq 0$ be the first non-zero
parameter from $\{\alpha_3, \alpha_4, \dots, \alpha_n\}$. Then
\[
\alpha_k=
\begin{cases}
0, &\text{if  \ $k \not\equiv s  \ \big(\bmod  (s-2)\big)$;}\\
(-1)^t C^{s-1}_{t+1} \alpha_{s}^{t+1}, &\text{if \ $k \equiv s \ \big(\bmod  (s-2)\big)$,}
\end{cases}
\]
where $3 \leq k \leq n$, $t = \frac{k-s}
{s-2}$
 and $C^{p}_{n} = \frac {1} {(p-1)n+1} \dbinom{pn}{n}$ is the $p$-th Catalan
number.
\end{lem}
\begin{proof} Since $\alpha_{s} \neq 0$, from the
equality \eqref{E:der}, we obtain $(a_1 - (s-2)a_0)\alpha_s =0$, and
consequently, $a_1 = (s-2)a_0$. Replacing $a_1 = (s-2)a_0$ we have
\[(s-k)a_0\alpha_k = \frac k 2 (s-2) a_0 \sum\limits_{j=4}^k\alpha_{j-1}\alpha_{k-j+3}, \quad k \geq s+1.\]

Since the algebra is non-characteristically nilpotent, we have
$a_0 \neq 0$ and
\begin{equation} \label{E:ak}
\alpha_k = \frac {k(s-2)} {2(s-k)}\sum\limits_{j=4}^k\alpha_{j-1}\alpha_{k-j+3}.
\end{equation}

We will prove the statement of the lemma by induction on $l= \lfloor \frac
{k-s} {s-2} \rfloor$, where $\lfloor x\rfloor$ is the integer part of $x$.

The base of induction $l=0$ is straightforward from the condition
of the lemma.

Let us suppose the induction hypothesis is true for $t<l$ and we will
prove it for $l = \lfloor \frac {k-s} {s-2}\rfloor$.

From equality \eqref{E:ak} we have that if $k \not\equiv s \ \big(\bmod
(s-2)\big)$, then one of the values of $j-1$ and $k-j+3$ are not
congruent by $ \bmod \, (s-2)$ with $s$, simultaneously. Otherwise, if
$j-1 \equiv s \ \big(\bmod  (s-2)\big)$ and $k-j+3 \equiv s \ \big(\bmod
(s-2)\big)$, then $k \equiv s \ \big(\bmod (s-2)\big)$, which is a
contradiction. Thus, by induction hypothesis, we have
$\alpha_{j-1}\alpha_{k-j+3} =0$ for any $j \ (4 \leq j \leq k)$,
which implies $\alpha_k = 0$ for any $k \not\equiv s \ \big(\bmod
(s-2)\big)$ with $\lfloor\frac {k-s} {s-2}\rfloor =l$.

If $k \equiv s \ \big(\bmod  (s-2)\big)$, i.e. $k = s+ (s-2)m$  then
\begin{align*}
\alpha_k &= \frac {(s+(s-2)m)(s-2)}
{2(s-s-(s-2)m)}\sum\limits_{j=4}^{s+(s-2)m}\alpha_{j-1}\alpha_{s+(s-2)m-j+3}\\
& {} = - \frac {s+(s-2)m} {2m}\sum\limits_{j=4}^{s+(s-2)m}\alpha_{j-1}\alpha_{s+(s-2)m-j+3}
= - \frac {s+(s-2)m}
{2t}\sum\limits_{j=s+1}^{(s-2)t+3}\alpha_{j-1}\alpha_{s+(s-2)t-j+3}.
\end{align*}
Changing $j-1 = s+ (s-2)j'$ and using induction hypothesis, we
obtain
\begin{align*}
\alpha_k & ={} - \frac {s+(s-2)t} {2t}\sum\limits_{j'=0}^{t-1}\alpha_{s+(s-2)j'}\alpha_{s+(s-2)(t-j'-1)}\\
&{} = - \frac {s+(s-2)t} {2t} (-1)^{t-1}\alpha_s^{t+1} \sum\limits_{j=0}^{t-1} C_{j+1}^{s-1}
C_{t-j}^{s-1}=  (-1)^{t}\alpha_s^{t+1} \Big(\frac {s+(s-2)t} {2t}
\sum\limits_{j=1}^{t} C_{j}^{s-1} C_{t+1-j}^{s-1}\Big).
\end{align*}
Applying formula \eqref{E:comb}, we derive
\[\alpha_k = (-1)^{t}\alpha_s^{t+1} C_{t+1}^{s-1},\]
where $3 \leq k \leq n$, $t = \frac{k-s} {s-2}$
 and $C^{p}_{n}$  is the $p$-th Catalan
number.
\end{proof}

Below, we present the classification of algebras obtained in Lemma
\ref{l4}.

\begin{thm} Let $L$ be a non-characteristically nilpotent filiform Leibniz algebra
of the family $F_1(\alpha_3, \alpha_4, \dots, \alpha_n,
\alpha_n)$. Then it is isomorphic to one of the following pairwise
non-isomorphic algebras:
\[F_1^{s}(\alpha_3, \alpha_4, \dots, \alpha_n, \alpha_n), \qquad 3 \leq s \leq n,\]
where
\[
\alpha_k=
\begin{cases}
0, &\text{if  \ $k \not\equiv s  \ \big(\bmod  (s-2)\big)$;}\\
(-1)^t C^{s-1}_{t+1}, &\text{if \ $k \equiv s \ \big(\bmod  (s-2)\big)$ \ for \ $  t = \dfrac{k-s} {s-2}$,}
\end{cases}
\]
$3 \leq k \leq n$ and  $C^{p}_{n} $ is the $p$-th Catalan number.
\end{thm}
\begin{proof}
From Lemma \ref{l4} we have
\[
\alpha_k=
\begin{cases}
0, &\text{if  \ $k \not\equiv s  \ \big(\bmod  (s-2)\big)$;}\\
(-1)^t C^{s-1}_{t+1} \alpha_{s}^{t+1}, &\text{if \ $k \equiv s \ \big(\bmod  (s-2)\big)$.}
\end{cases}
\]
From Theorem \ref{thm2}, we have the isomorphism criterion
\[\alpha'_s=\frac{1}{A^{s-1}} (A+B)\alpha_s.\]

Putting $B = \frac {A^{s-1}} {\alpha_s} -A$, we get $\alpha'_s=1$.
Thus, without loss of generality we can assume $\alpha_s=1$, then
\[
\alpha_k=
\begin{cases}
0, &\text{if  \ $k \not\equiv s  \ \big(\bmod  (s-2)\big)$;}\\
(-1)^t C^{s-1}_{t+1}, &\text{if \ $k \equiv s \ \big(\bmod  (s-2)\big)$.}
\end{cases}
\]
\end{proof}

\subsection{Non-characteristically nilpotent filiform Leibniz algebras of the family
$F_2(\beta_3,\beta_4, \dots, \beta_n, \gamma )$}

\

Now we consider  algebras of the family $F_2(\beta_3,\beta_4,
\dots, \beta_n, \gamma )$. Similar to the above section, firstly we
describe the derivations of such algebras.

\begin{prop}\label{pr2} Any derivation of a filiform Leibniz
algebra of the family $F_2(\beta_3,\beta_4, \dots, \beta_n, \gamma)$ has the form:
\[\begin{pmatrix}
a_0&a_1&a_2&a_3&\dots&a_{n-2}&a_{n-1}&a_n\\
0&b_1&0&0&\dots&0&-a_1\gamma & b_n\\
 0& 0& 2a_0&a_2+a_1\beta_3&\dots&a_{n-3}+a_1\beta_{n-2}&
a_{n-2}+a_1\beta_{n-1}&a_{n-1}+a_1\beta_n\\
 0& 0& 0& 3a_0& \dots& a_{n-4}+2a_1\beta_{n-3}& a_{n-3}+2a_1\beta_{n-2}&a_{n-2}+2a_1\beta_{n-1}\\
\vdots& \vdots& \vdots& \vdots&\vdots& \vdots&\vdots &\vdots \\
0& 0& 0& 0&\dots& (n-2)a_0&a_2+(n-3)a_1\beta_3 & a_3+(n-3)a_1\beta_4\\
 0& 0& 0& 0&\dots& 0&(n-1)a_0 & a_2+(n-2)a_1\beta_3\\
0& 0& 0& 0&\dots& 0& 0& na_0
\end{pmatrix}\]
where
\begin{align} \label{E:bk}
\gamma (2b_1 - na_0) & =0, \qquad  \qquad \beta_3(b_1-2a_0)=0, \notag \\
\beta_k(b_1-(k-1)a_0)& =\frac k 2 a_1\sum\limits_{j=4}^{k}\beta_{j-1}\beta_{k-j+3}, \qquad 4 \leq k \leq n-1, \\
\beta_n(b_1-(n-1)a_0) & = - a_1\gamma + \frac n 2 a_1\sum\limits_{j=4}^{n}\beta_{j-1}\beta_{n-j+3}.\ \notag
\end{align}
\end{prop}
\begin{proof} Let $L$ be a filiform Leibniz algebra from the
second family and let $d: L \rightarrow L$ be a derivation of $L$.

We set
\[d(e_0) = \sum\limits_{k=0}^na_ke_k, \quad d(e_1) = \sum\limits_{k=0}^nb_ke_k.\]

From the property of the derivation $d$ one has
\begin{align*}
d(e_2) &= d([e_0, e_0]) = [d(e_0), e_0] +  [e_0, d(e_0)] =
\big[\sum\limits_{k=0}^na_ke_k, e_0\big]+\big[e_0, \sum\limits_{k=0}^na_ke_k\big]\\
&{} =a_0e_2 + \sum\limits_{k=3}^na_{k-1}e_k + a_0e_2 + a_1\sum\limits_{k=3}^n\beta_{k}e_k
= 2a_0e_2 + \sum\limits_{k=3}^n(a_{k-1}+a_1\beta_{k})e_k.
\end{align*}

By induction it is not difficult to obtain
\[d(e_i)=ia_0e_i + \sum\limits_{k=i+1}^n\big(a_{k+1-i}+(i-1)a_1\beta_{k+2-i}\big)e_k, \quad 2 \leq i \leq n.\]

Consider the property of the derivation
\begin{align*}
d([e_1, e_0]) &= [d(e_1), e_0] +  [e_1, d(e_0)] =\big[\sum\limits_{k=0}^nb_ke_k, e_0\big]+\big[e_1, \sum\limits_{k=0}^na_ke_k\big]  \\
&{}= b_0e_2+\sum\limits_{k=3}^{n}b_{k-1}e_k + a_1 \gamma e_n.
\end{align*}

On the other hand
\[d([e_1, e_0]) = 0.\]

Consequently, $b_0 = b_2=b_4 = \dots = b_{n-2} = 0, \quad b_{n-1}
= -a_1\gamma$.

From the chain of equalities
\begin{align*}
n a_0\gamma e_n & =d(\gamma e_n) =d([e_1, e_1]) = [d(e_1), e_1] +  [e_1, d(e_1)] \\
&{}=[b_1e_1 + b_{n-1}e_{n-1} + b_ne_n, e_1]+[e_1, b_1e_1 + b_{n-1}e_{n-1} + b_ne_n] = 2b_1\gamma e_n,
\end{align*}
we get $(2b_1- na_0)\gamma =0$.

Using Lemma \ref{lem2} and derivation property, we obtain
\begin{align*}
d([e_0, e_1]) & = [d(e_0), e_1] +  [e_0, d(e_1)]\\
&{}=(a_0+b_1)\beta_3e_3+\sum\limits_{k=4}^n(a_0+b_1)\beta_ke_k+a_1\gamma e_n +
\sum\limits_{k=4}^{n}\Big(\sum\limits_{j=4}^ka_{j-2}\beta_{k-j+3}\Big)e_k.
\end{align*}

On the other hand
\begin{align*}
d([e_0, e_1]) & = d\Big(\sum\limits_{k=3}^n\beta_{k}e_k\Big) = \sum\limits_{k=3}^n\beta_{k}d(e_k) \\
&{}=3a_0\beta_3e_3+\sum\limits_{k=4}^nka_0\beta_{k}e_k
+\sum\limits_{k=4}^n\Big(\sum\limits_{j=4}^ka_{j-2}\beta_{k-j+3}\Big)e_k
+a_1\sum\limits_{k=4}^n\frac k 2\Big(\sum\limits_{j=4}^k
\beta_{j-1}\beta_{k-j+3}\Big)e_k.
\end{align*}

Comparing the coefficients at the basic elements we deduce
\begin{align*}
\beta_3(b_1-2a_0) & =0,\\
\beta_k(b_1-(k-1)a_0) & =\frac k 2 a_1\sum\limits_{j=4}^k
\beta_{j-1}\beta_{k-j+3}, \qquad \qquad 4 \leq k \leq n-1,\\
\beta_n(b_1-(n-1)a_0)& = -a_1\gamma + \frac n 2 a_1\sum\limits_{j=4}^n
\beta_{j-1}\beta_{n-j+3}.
\end{align*}

Considering the properties of the  derivation for products $d([e_i,
e_1]), \ 2 \leq i \leq n-2$, we already  get  the obtained restrictions.
\end{proof}

From Proposition \ref{pr2} it is obvious that if there exists the
pair $a_0, b_1$ such that $(a_0, b_1) \neq (0, 0)$ and the
restriction \eqref{E:bk} holds, then a filiform Leibniz algebra is
non-characteristically nilpotent, otherwise is characteristically
nilpotent.

It is known that a naturally graded filiform Leibniz algebra of
the second family (an algebra with $\gamma =0$ and $\beta_i =0, \
3 \leq i \leq n$) is non-characteristically nilpotent \cite{CLOK2,Omi1}.

\begin{thm}\label{t1} Let $\gamma \neq 0$ and $n$ be odd. If there
there exist $\beta_i \neq 0, \ 3 \leq i \leq n-1$, then a filiform
Leibniz algebra from $F_2(\beta_3, \beta_4, \dots, \beta_n,
\gamma)$ is characteristically nilpotent.
\end{thm}
\begin{proof} If $\gamma \neq 0$, then since $\gamma(2b_1 - na_0) =0$, this
implies that $b_1 = \frac{n a_0} {2} $ and we get, that the restrictions
\eqref{E:bk} have the form
\begin{align*}
\frac{(n-4)a_0} {2}\beta_3 & =0,\\
\frac{n-2k+2} 2 \beta_ka_0& =\frac k 2 a_1\sum\limits_{j=4}^{k}\beta_{j-1}\beta_{k-j+3}, \qquad \qquad  4 \leq k \leq n-1,\\
\frac{-n+2} 2 \beta_na_0 & = - a_1\gamma + \frac n 2 a_1\sum\limits_{j=4}^{n}\beta_{j-1}\beta_{n-j+3}.
\end{align*}

If there exist $\beta_k \neq 0, \ 3 \leq k \leq n-1$, then for the
first non-zero $\beta_k \neq 0$, we get \[(n-2k+2)a_0 \beta_k =
0.\] Since $n$ is odd, we conclude $a_0=b_1=0$.
\end{proof}

Let us clarify the situation when $\beta_i =0$ for $3 \leq i \leq
n-1$.

\begin{thm} Let $\gamma \neq 0$ and $n$
be odd. Then any non-characteristically nilpotent filiform Leibniz
algebra of the second class is isomorphic to the algebra \[F_2(0, 0,
\dots, 0, 0, 1).\]
\end{thm}
\begin{proof}
Theorem \ref{t1} implies that if $n$ is odd, then
 a non-characteristically nilpotent filiform Leibniz algebra of the
second family has the form $F_2(0, 0, \dots, 0, \beta_n, \gamma)$,
where
\[\frac{-n+2} 2\beta_n a_0 = - a_1\gamma.\]

Since $\gamma \neq 0$, then for any $\beta_n$ putting $a_0 \neq 0$
and $a_1 = \frac{(n-2)\beta_n a_0} {2 \gamma}$,
%$b_{n-1}=\frac{(-n+2)\beta_n a_0} {2}$,
we have a non-singular derivation.

From Theorem \ref{thm2}, we derive $\beta'_k = 0, \ 3 \leq k \leq
n-1$, and the isomorphism criterion is

\[\gamma'=\frac{D^2}{A^{n}} \gamma, \quad
\beta'_n=\frac{BD \gamma}{A^{n}}\beta_n+ \frac{D}{A^{n-1}}\beta_n.
\]
Since $\gamma \neq 0$, putting $D=\sqrt{\frac{A^{n}} {\gamma}}$, we get
\[\gamma'=1, \quad
\beta'_n=\frac {B\gamma+A\beta_n} {\sqrt{\gamma A^{n}}}.
\]
Setting $B = - \frac {A\beta_n}{\gamma}$, we have $\beta_n'=0$ and so we
obtain the algebra  $F_2(0, 0, \dots, 0, 1)$.
\end{proof}

Now we investigate the even $n$ case.
\begin{thm}\label{t3}
 Let $\gamma \neq 0$ and $n$ be even.
 If there exist $\beta_k \neq 0, \ 3 \leq k \leq n-1, \ k
\neq \frac{n+2} 2$, then an algebra of the family $F_2(\beta_3,
\beta_4, \dots, \beta_n, \gamma)$ is characteristically nilpotent.
\end{thm}
\begin{proof}
Analogously to the proof of Theorem \ref{t1} \end{proof}

\begin{thm}
 Let $L$ be a non-characteristically nilpotent filiform
Leibniz algebra from the family $F_2(\beta_3,\beta_4, \dots, \beta_n,
\gamma)$. If $\gamma \neq 0$ and  $n$ is even, then it is
isomorphic to one of the following pairwise non-isomorphic algebras:
\[F^1_2(0, \dots, 0, \beta_{\frac{n+2}{2}}, 0,  \dots, 0, 0, 1), \
F^2_2(0, \dots, 0, \sqrt{\frac{2}{n}}, 0,  \dots, 0, 1, 1).\]
\end{thm}
\begin{proof} Let $L$ be a non-characteristically nilpotent filiform
Leibniz algebra, then by Theorem \ref{t3} we have $\beta_k = 0, 3
\leq k \leq n-1, \ k \neq \frac{n+2} 2$ and
\begin{equation} \label{E:bna0}
\frac{-n+2} 2 \beta_na_0= - a_1\gamma + \frac n 2 \beta_{\frac{n+2}{2}}^2.
\end{equation}

Since $\gamma \neq 0$, then for any values of
$\beta_{\frac{n+2}{2}}$ and $\beta_n$ there exist $a_0, a_1 \ (a_0
\neq 0)$ such that the restriction \eqref{E:bna0} is held. Therefore,
a non-characteristically nilpotent filiform Leibniz algebra of the
second family has the form \[F_2(0, \dots, 0,
\beta_{\frac{n+2}{2}}, 0, \dots, 0, \beta_n, \gamma).\]

By Theorem \ref{thm2} we have the isomorphism criterion
\[\beta'_k =0, \quad 3 \leq k \leq n-1, \  \   k \ \neq \frac {n+2} 2,\]
\[\gamma'=\frac{D^2}{A^{n}} \gamma, \quad
\beta'_{\frac{n+2}{2}}=\frac{D}{A^{\frac{n}2}}\beta_{\frac{n+2}{2}},
\quad
\beta'_n=\frac{BD\gamma}{A^{n}}+\frac{D}{A^{n-1}}\Big(\beta_n -
\frac {nB} {2A} \beta^2_{\frac{n+2}{2}}\Big).
\]

Putting $D=\frac{A^{\frac{n}2}}{\sqrt {\gamma}}$, we get
\[\gamma'=1, \quad
\beta'_{\frac{n+2}{2}}=\frac{\beta_{\frac{n+2}{2}}}{\sqrt{\gamma}},
\quad \beta'_n= \frac{1}{\sqrt{A^n\gamma}}\Big((\gamma - \frac
{n} {2} \beta^2_{\frac{n+2}{2}})B + A\beta_n \Big).\]

It is not difficult to check that $\gamma' - \frac{n} 2
\beta'^2_{\frac {n+2} 2} = \frac{D^2} {A^n} (\gamma - \frac{n} 2
\beta^2_{\frac {n+2} 2})$.

If $\gamma \neq \frac{n} 2 \beta^2_{\frac {n+2} 2}$, then putting
$B =\frac {2 A \beta_n} {2\gamma - n \beta^2_{\frac {n+2} 2}}$, we
have $\beta'_n = 0$, and so obtain the algebra \[F^1_2(0, \dots, 0,
\beta_{\frac{n+2}{2}}, 0,  \dots, 0, 0, 1),  \quad
\beta_{\frac{n+2}{2}} \neq \sqrt{\frac{2}{n}}.\]

If $\gamma = \frac{n} 2 \beta^2_{\frac {n+2} {2}}$, then we have
$\beta'_{\frac{n+2}{2}} = \sqrt{\frac {2} {n} }, \quad \beta'_n =
\frac {\beta_n} {\sqrt{A^{n-2}\gamma}}$.

For $\beta_n =0$ we obtain $\beta'_n =0$ and the algebra $F^1_2(0,
\dots, 0,  \sqrt{\frac{2}{n}}, 0,  \dots, 0, 0, 1)$.

For $\beta_n \neq 0$ we set $A= \sqrt[\uproot{4}n-2]{\frac {\beta_n^2}
{\gamma}}$. Thus, we get $\beta_n' = 1$ and the algebra $F^2_2(0,
\dots, 0,  \sqrt{\frac{2}{n}}, 0,  \dots, 0, 1, 1)$.
\end{proof}

Let us investigate the case $\gamma = 0$.

\begin{thm} Let $L$ be a non-characteristically nilpotent filiform
Leibniz algebra from $F_2(\beta_3\beta_4, \dots, \beta_n,
\gamma)$. If $\gamma =0$, then it is isomorphic to one of the
following pairwise non-isomorphic algebras:
\[F_2^j(0,\dots, 0, \overset{j}{1}, 0 \dots, 0 ,
0), \qquad 3 \leq j \leq n.\]
\end{thm}
\begin{proof}
The restrictions \eqref{E:bk} under the conditions of the theorem have the
form
\begin{align*}
\beta_3(b_1-2a_0)& =0,\\
\beta_k(b_1-(k-1)a_0) & =\frac k 2 a_1\sum\limits_{j=4}^{k}\beta_{j-1}\beta_{k-j+3}, \qquad 4 \leq k \leq n.
\end{align*}

Let $\beta_j$ be the first non-zero parameter, i.e.  $\beta_i=0$
for $3 \leq i \leq j-1$ and $\beta_j\neq 0$. Then we have
$\beta_j(b_1-(j-1)a_0)=0$, which implies $b_1 = (j-1)a_0$. Since
the algebra is non-characteristically nilpotent, the coefficient $a_0
\neq 0$. Therefore, from the restrictions \eqref{E:bk}, we derive
\begin{align}\label{E:ba}
\beta_k & =0, \qquad  \qquad  \qquad  \qquad j+1 \leq k \leq 2j-3, \notag\\
(2-j)a_0 \beta_{2j-2}&=(j-1)a_1\beta^2_{j}, \\
\beta_k & =0, \qquad  \qquad  \qquad \qquad  k \neq t(j-2)+2, \notag \\
(t-1)(2-j)a_0 \beta_{t(j-2)+2} & = \frac
{t(j-2)+2}{2}a_1
\sum\limits_{i=1}^{t-1}\beta_{i(j-2)+2}\beta_{(t-i)(j-2)+2}\,,  \notag
\end{align}
where $t \geq 3, k \leq n$.

From Theorem \ref{thm2} we get the isomorphism criterion
\begin{align*}
\beta_k' &= 0, \quad 3 \leq k \leq j-1,   \quad  & \beta_j' &= \frac D {A^{j-1}} \beta_j, \\
 \beta_k' & = 0, \quad j+1 \leq k \leq 2j-3 ,  \quad & \beta_{2j-2}' & = \frac D {A^{2j-3}} \Big(\beta_{2j-2} - \frac {(j-1)B}{A}
\beta^2_j\Big).
\end{align*}

Putting $D = \frac{A^{j-1}}{\beta_j}$, $B=\frac{A
\beta_{2j-2}}{(j-1)\beta_j^2}$, we obtain $\beta_j'=1$,
$\beta_{2j-2}'=0$. Thus, if $\beta_i=0$ for $3 \leq i \leq j-1$
and $\beta_j\neq 0$, then, without loss of generality, we can
suppose $\beta_j=1$ and $\beta_{2j-2}=0$.

The restrictions \eqref{E:ba} imply $a_1=0$ and $(t-1)(2-j)a_0
\beta_{t(j-2)+2} = 0, \ t \geq 3$, which leads to $\beta_k = 0$ for
all $k \neq j$. Thus, we obtain the algebras $F_2^j(0, \dots, 0,
\overset{j}{1}, 0, \dots, 0, 0), \ 3 \leq j \leq n$.
\end{proof}

\subsection{Non-characteristically nilpotent filiform Leibniz algebras of the family
$F_3(\theta_1, \theta_2, \theta_3)$}

\

Since non-characteristically nilpotent filiform Lie algebras are
described in \cite{GoHa}, we will classify them only in the non-Lie case.

Let $L$ be a filiform non-Lie Leibniz algebra of the third family.
Put
\[[e_i,e_1]=-[e_1, e_i] = \beta_{i,i+2}e_{i+2} +
\beta_{i,i+3}e_{i+3} + \dots + \beta_{i,n}e_{n}, \quad 2\leq i
\leq {n-2}.\]

Using the Leibniz identity it is not difficult to obtain the
following equality:
\[[e_i, e_k] = \sum\limits_{j=i+k+1}^n\Bigg(\sum\limits_{t=0}^{k-1} (-1)^t \ \binom{k-1}{t} \
\beta_{i+t, j+1-k+t}\Bigg)e_j,
  \quad 2 \leq k, \quad k \leq i \leq n-k-1.\]
Since $[e_k, e_k] = 0$, we have $\sum\limits_{t=0}^{k-1} (-1)^t \ \dbinom{k-1}{t} \
 \beta_{k+t,k+t+2} =0, \ 2 \leq k \leq n$.

\begin{prop}\label{third class} Let $L$ be a non-characteristically nilpotent non-Lie filiform Leibniz
algebra from the family $F_3(\theta_1, \theta_2, \theta_3)$. Then
\[\beta_{i,j} =0, \quad 2 \leq i \leq n-2,  \ i+2 \leq j \leq n,\]
i.e. $[e_i, e_j] =0$, for $i+j \geq 3$.
\end{prop}
\begin{proof} Let $L$ be a non-characteristically nilpotent non-Lie filiform Leibniz
algebra from the family $F_3(\theta_1, \theta_2, \theta_3)$ and
let $d$ be a derivation of $L$.

Put
\[d(e_0) = \sum\limits_{k=0}^na_ke_k, \quad d(e_1) = \sum\limits_{k=0}^nb_ke_k.\]

Similarly as above we establish $b_0=0$.

From the property of the derivation we have
\[d(e_i) = \big((i-1)a_0+b_1\big)e_i \ (\bmod  \ L^{i+1}), \quad 2 \leq i \leq n.\]

Consider the equalities
\begin{align*}
d([e_0, e_0]) & = [d(e_0), e_0] +  [e_0, d(e_0)] =\big[\sum\limits_{k=0}^na_ke_k, e_0\big]+\big[e_0, \sum\limits_{k=0}^na_ke_k\big]  \\
& {} = a_0[e_0, e_0] + a_1[e_1, e_0]+ a_0[e_0, e_0] + a_1[e_0, e_1] = (2a_0\theta_{1} + a_1\theta_{2})e_n.
\end{align*}

On the other hand,
\[d([e_0, e_0]) = \theta_{1}d(e_n) = \theta_{1}\big((n-1)a_0+b_1\big)e_n.\]

Consequently,
\[\theta_{1}\big((n-3)a_0+b_1\big) = a_1\theta_{2}.\]

Consider the property of the derivation for the product $[e_0, e_1]$,
\begin{align*}
&d([e_0, e_1]) = [d(e_0), e_1] +  [e_0, d(e_1)] =
a_0[e_0, e_1] +  a_1[e_1, e_1] + \big[\sum\limits_{k=2}^na_ke_k, e_1\big] + b_1[e_0, e_1] +
\big[e_0, \sum\limits_{k=2}^nb_ke_k\big]\\
&{} = -a_0[e_1, e_0] - a_1[e_1, e_1] - \big[e_1,
\sum\limits_{k=2}^na_ke_k\big] +a_0\theta_2e_n + 2a_1\theta_3e_n
-b_1[e_1, e_0] - \big[\sum\limits_{k=2}^nb_ke_k, e_0\big] +
b_1\theta_2e_n \\
&{} = - \big[e_1,
\sum\limits_{k=0}^na_ke_k\big] - \big[\sum\limits_{k=1}^nb_ke_k, e_0\big]
+(a_0\theta_2 + b_1\theta_2 + 2a_1\theta_3)e_n= -d(e_2)+
(a_0\theta_2 + b_1\theta_2 + 2a_1\theta_3)e_n.
\end{align*}

On the other hand,
\[d([e_0, e_1]) = d(-e_2 + \theta_2e_n) = - d(e_2) + \theta_2d(e_n) =
- d(e_2) + \theta_2\big((n-1)a_0 + b_1\big)e_n.\]

Therefore, $2a_1\theta_{3}= (n-2)a_0\theta_{2}$.

Similarly, from
\[d([e_1, e_1]) = \big[\sum\limits_{k=1}^nb_ke_k, e_1\big] +  \big[e_1, \sum\limits_{k=1}^nb_ke_k\big] =[b_1e_1, e_1]+[e_1, b_1e_1]  = 2b_1\theta_{3}e_n\]
and
\[d([e_1, e_1]) = d(\theta_{3}e_n) =  \theta_{3} \big((n-1)a_0 + b_1\big)e_n\]
we conclude that $\theta_{3}\big((n-1)a_0 - b_1\big) =0$.

Thus, we obtain
\[\theta_{1}\big((n-3)a_0+b_1\big) = a_1\theta_{2}, \quad 2a_1\theta_{3}=
(n-2)a_0\theta_{2}, \quad \theta_{3}\big((n-1)a_0 - b_1\big) =0.\]

Note that for a non-Lie Leibniz algebra we have  $(\theta_{1},
\theta_{2}, \theta_{3}) \neq (0, 0, 0)$.

If $\theta_{3} \neq 0$, then $b_1 = (n-1)a_0, \ a_1 = \frac
{(n-2)a_0\theta_{2}} {2\theta_{3}}$ and $(2n-4)a_0\theta_1 = \frac
{(n-2)a_0\theta_{2}^2} {2\theta_{3}}$.

Note that $a_0 \neq 0$ (since $L$ is non-characteristically
nilpotent). Therefore, we deduce
\[\theta_1 = \frac {\theta_{2}^2} {4\theta_{3}}.\]

If $\theta_{3} = 0$, then $\theta_{1}\big((n-3)a_0+b_1\big) =
a_1\theta_{2}, \ (n-2)a_0\theta_{2} =0$.

If $\theta_{2} \neq 0$, then $a_0 =0$, and $\theta_{1}b_1 =
a_1\theta_{2}$.

If $\theta_{2} = 0$, then $\theta_{1} \neq 0$, and $b_1 =
-(n-3)a_0$.

Thus, on the behavior of the parameters $\theta_{1}, \theta_{2}$
and $\theta_{3}$ we obtain the following equalities: \[\ b_1 =
(n-1)a_0, \quad a_0 =0, \quad b_1 = -(n-3)a_0.\]

Now we shall prove $\beta_{i,j} =0, \quad 4 \leq j \leq n, \ 2
\leq i \leq j-2$, by induction on $j$ for any values of $i$.

Consider the property of the derivation for the product $[e_2, e_1]$,
\begin{align*}
d([e_2, e_1]) &= [d(e_2), e_1] +  [e_2, d(e_1)] =
[(a_0+b_1)e_2 + x_3, e_1]+\big[e_2, \sum\limits_{k=1}^nb_ke_k\big] \\
&{}=(a_0+2b_1)\beta_{2,4}e_{4} + x_{5}.
\end{align*}

On the other hand,
\[d([e_2, e_1]) =\beta_{2,4}d(e_{4}) + \beta_{2,5}d(e_{5}) + \dots + \beta_{2,n}d(e_{n}) =
(3a_0+b_1)\beta_{2,4}e_4+ y_5,\] where $x_3 \in L^3$ and $x_5, y_5
\in L^5$.

Comparing the coefficients at the basic element $e_4$, we obtain
\[\beta_{2,4}(b_1-2a_0) = 0.\]

Since $b_1 = (n-1)a_0$ or $a_0 =0$ or $b_1 = -(n-3)a_0$ and $(a_0,
b_1) \neq (0,0)$, we have $\beta_{2,4}=0$. Thus, we proved the
statement of the proposition for $j=4$.

Let the induction hypothesis be true for $j\leq k \leq n-1$, i.e.
$\beta_{i,j} = 0$ for $4 \leq j \leq k, \ 2 \leq i \leq j-2$,
which implies $[L^{i_0}, e_1] \subseteq L^{k+1}$, $[L^{i_0+1},
e_1] \subseteq L^{k+2}$. We will prove $\beta_{i,k+1} = 0$ for $2
\leq i \leq k-1$.

Let us suppose the contrary, i.e. there exists $i$ such that
$\beta_{i,k+1} \neq 0$. Let $i_0$ be the greatest number among
indexes such that $\beta_{i,k+1} \neq 0$.

Again, consider the property of the derivation
\[d([e_{i_0}, e_1]) = [d(e_{i_0}), e_1] +  [e_{i_0}, d(e_1)]
=\big((i_0-1)a_0+2b_1\big)\beta_{i_0,k+1}e_{k+1} + x_{k+2}.\]

On the other hand,
\[d([e_{i_0}, e_1]) =\sum_{j=k+1}^{n}\beta_{i_0,j}d(e_{j})=
(ka_0+b_1)\beta_{{i_0},k+1}e_{k+1}+ y_{k+2},\] where  $x_{k+2},
y_{k+2} \in L^{k+2}$.

Comparing the coefficients at the basic element $e_{k+1}$, we get
$\beta_{{i_0},k+1}(b_1-(k+1-{i_0})a_0) = 0$. Since $2 \leq {i_0}
\leq k-1$, we have $2 \leq k+1-{i_0} \leq k-1$.

Taking into account $k+1\leq n$ and the correctness of  one of the
following conditions:  \[b_1 = (n-1)a_0, \quad  a_0 =0, \quad  b_1 =
-(n-3)a_0,\] we deduce $\beta_{{i_0},k+1}=0$ for $2 \leq i \leq
k-1$, which is a contradiction with the assumption $\beta_{i,k+1} \neq
0$. Thus, $\beta_{i,k+1} = 0$ and we proved that $\beta_{{i},j} =
0$ for all $i,j$.
\end{proof}

\begin{rem} The proof of Proposition \ref{third class} shows
that the cases $\alpha =0$ and $\alpha =1$ are proved analogously.
\end{rem}

Proposition \ref{third class} implies that the table of
multiplication of a non-characteristically nilpotent filiform
Leibniz algebra from the third family has the form:
\[F_3(\theta_1, \theta_2,\theta_3)=\left\{\begin{array}{lll} [e_i,e_0]=e_{i+1}, &
1\leq i \leq {n-1},\\[1mm]
[e_0,e_i]=-e_{i+1}, & 2\leq i \leq {n-1}, \\[1mm]
[e_0,e_0]=\theta_1e_n, &   \\[1mm]
[e_0,e_1]=-e_2+\theta_2e_n, & \\[1mm]
[e_1,e_1]=\theta_3e_n, &  \\[1mm]
[e_i,e_{n-i}]=-[e_{n-i},e_i]=\alpha (-1)^{i}e_n, & 1\leq i\leq
n-1.
\end{array} \right.\]

For such algebras in \cite{OmRa} it is obtained the isomorphism
criterion:
\begin{equation}\label{E:crit}
\theta_1' = \frac{A_0^2 \theta_1 + A_0A_1 \theta_2 + A_1^2
\theta_3}{A_0^{n-1}B_1}, \qquad
\theta_2'=\frac{A_0\theta_2+2A_1\theta_3}{A_0^{n-1}}, \qquad
\theta_3'=\frac{B_1\theta_3}{A_0^{n-1}}.
\end{equation}

\begin{thm} Let $L$ be a non-characteristically nilpotent filiform
Leibniz algebra from $F_3(\theta_1, \theta_2,\theta_3)$. Then it
is isomorphic to one of the following pairwise non-isomorphic
algebras
\[F_3^1(1,0,0), \quad F_3^2(0,1,0), \quad F_3^3(0,0,1).\]
\end{thm}
\begin{proof}
Consider several cases.

\textbf{Case 1.} Let $\theta_{3} = 0$ and $\theta_{2} = 0$. Then
$\theta_1 \neq 0$ and
\[\theta_1' = \frac{\theta_1 }{A_0^{n-3}B_1}, \quad
\theta_2'=\theta_3'=0.
\]

Putting $B_1 = \frac{\theta_1}{A_0^{n-3}}$, we have $\theta_1'=1$
and obtain the algebra $F_3^1(1,0,0)$.

\textbf{Case 2.} Let $\theta_{3} = 0$ and $\theta_{2} \neq 0$. Then
we have
\[\theta_1' = \frac{A_0 \theta_1 + A_1 \theta_2}{A_0^{n-2}B_1}, \quad
\theta_2'=\frac{\theta_2}{A_0^{n-2}}, \quad \theta_3'=0.
\]

Putting $A_0 = \sqrt[\uproot{3}n-2]{\theta_2}$, $A_1 = -\frac{A_0
\theta_1}{\theta_2}$, we have $\theta_1'=0, \theta_2'=1$ and
obtain the algebra $F_3^2(0, 1, 0)$.

\textbf{Case 3.} Let $\theta_3 \neq 0$. Then similarly as in the
proof of Proposition \ref{third class} we conclude $\theta_1 =
\frac {\theta_{2}^2} {4\theta_{3} }$. Then in the isomorphism
criterion \eqref{E:crit} putting $B_1 = \frac{A_0^{n-1}}{\theta_3}$, $A_1 =
- \frac{A_0 \theta_2}{2\theta_3}$, we get
\[\theta_1'=\theta_2'=0, \quad \theta_3'=1.\]

Thus, in this case we have the algebra $F_3^3(0, 0, 1)$.

The non-characteristically property of obtained algebras follows
from the proof of Proposition \ref{third class}.
\end{proof}

\section*{ Acknowledgements}

The second and third authors were supported by MICINN, grant MTM 2009-14464-C02
(Spain) (European FEDER support included), and by Xunta de Galicia,
grant Incite09 207 215PR.

%\bibliographystyle{elsart-num-sort}
%\bibliography{biblio12}

\begin{thebibliography}{99}
\bibitem{AAOK}
S.~Albeverio, S.~A. Ayupov, B.~A. Omirov, A.~K. Khudoyberdiyev,
  {$n$}-dimensional filiform {L}eibniz algebras of length {$(n-1)$} and their
  derivations, J. Algebra 319~(6) (2008) 2471--2488.

\bibitem{AyOm2}
S.~A. Ayupov, B.~A. Omirov, On some classes of nilpotent {L}eibniz algebras,
  Siberian Math. J. 42~(1) (2001) 15--24.

  \bibitem{CLOK1}
J.~M. Casas, M.~Ladra, B.~A. Omirov, I.~K. Karimjanov, Classification of
  solvable {L}eibniz algebras with null-filiform nilradical, arXiv:1202.5275.

\bibitem{CLOK2}
J.~M. Casas, M.~Ladra, B.~A. Omirov, I.~K. Karimjanov, Classification of
  solvable {L}eibniz algebras with naturally graded filiform nilradical,
  arXiv:1203.4772.

\bibitem{CaNu}
F.~J. Castro-Jim{\'e}nez, J.~N{\'u}{\~n}ez~Vald{\'e}s, On characteristically
  nilpotent filiform {L}ie algebras of dimension {$9$}, Comm. Algebra 23~(8)
  (1995) 3059--3071.


\bibitem{DiLi}
J.~Dixmier, W.~G. Lister, Derivations of nilpotent {L}ie algebras, Proc. Amer.
  Math. Soc. 8 (1957) 155--158.

\bibitem{GoJiKh}
J.~R. G{\'o}mez, A.~Jimen{\'e}z-Merch{\'a}n, Y.~Khakimdjanov, Low-dimensional
  filiform {L}ie algebras, J. Pure Appl. Algebra 130~(2) (1998) 133--158.

\bibitem{GoOm}
J.~R. G{\'o}mez, B.~A. Omirov, On classification of complex filiform {L}eibniz
  algebras, arXiv:0612735v1.

\bibitem{Gou}
H.~W. Gould, Combinatorial identities: {T}able II: {A}dvanced techniques for
  summing finite series (2010).

\bibitem{GoHa}
M.~Goze, Y.~Hakimjanov, Sur les alg\`ebres de {L}ie nilpotentes admettant un
  tore de d\'erivations, Manuscripta Math. 84~(2) (1994) 115--124.

\bibitem{HiPe}
P.~Hilton, J.~Pedersen, Catalan numbers, their generalization, and their uses,
  Math. Intelligencer 13~(2) (1991) 64--75.


\bibitem{Jac2}
N.~Jacobson, A note on automorphisms and derivations of {L}ie algebras, Proc.
  Amer. Math. Soc. 6 (1955) 281--283.

\bibitem{Kha1}
Yu.~B. Khakimdzhanov, Characteristically nilpotent {L}ie algebras,
  Math. USSR, Sb. 70~(1)(1991)  65--78.

\bibitem{Kha2}
Yu.~B. Khakimdzhanov, Vari\'et\'e des lois d'alg\`ebres de {L}ie nilpotentes,
  Geom. Dedicata 40~(3) (1991) 269--295.


\bibitem{LaRiTu}
M.~Ladra, I.~M. Rikhsiboev, R.~M. Turdibaev, Automorphisms and derivations of
  {L}eibniz algebras, arXiv:1103.4721.

\bibitem{LeTo}
G.~Leger, S.~T{\^o}g{\^o}, Characteristically nilpotent {L}ie algebras, Duke
  Math. J. 26 (1959) 623--628.

\bibitem{Lod}
J.-L. Loday, Une version non commutative des alg\`ebres de {L}ie: les
  alg\`ebres de {L}eibniz, Enseign. Math. (2) 39~(3-4) (1993) 269--293.

\bibitem{NdWi}
J.~C. Ndogmo, P.~Winternitz, Solvable {L}ie algebras with abelian nilradicals,
  J. Phys. A 27~(2) (1994) 405--423.


\bibitem{Omi2}
B.~A. Omirov, Classification of eight-dimensional complex filiform {L}eibniz
  algebras, Uzbek. Mat. Zh.~(3) (2005) 63--71.

\bibitem{Omi1}
B.~A. Omirov, On derivations of filiform {L}eibniz algebras, Math. Notes
  77~(5--6) (2005) 677–--685.


\bibitem{OmRa}
B.~A. Omirov, I.~S. Rakhimov, On {L}ie-like complex filiform {L}eibniz
  algebras, Bull. Aust. Math. Soc. 79~(3) (2009) 391--404.


\bibitem{RaSo}
I.~S. Rakhimov, J.~Sozan, Description of nine dimensional complex filiform
  {L}eibniz algebras arising from naturally graded non {L}ie filiform {L}eibniz
  algebras, Int. J. Algebra 3~(17-20) (2009) 969--980.

\bibitem{Rik}
I.~M. Rikhsiboev, Classification of seven-dimensional complex filiform
  {L}eibniz algebras ({R}ussian), Uzbek. Mat. Zh.~(3) (2004) 57--61.

\bibitem{SnWi}
L.~{\v{S}}nobl, P.~Winternitz, A class of solvable {L}ie algebras and their
  {C}asimir invariants, J. Phys. A 38~(12) (2005) 2687--2700.


\bibitem{WaLiDe}
Y.~Wang, J.~Lin, S.~Deng, Solvable {L}ie algebras with quasifiliform
  nilradicals, Comm. Algebra 36~(11) (2008) 4052--4067.
\end{thebibliography}

\end{document}